\documentclass[12pt]{article}
\usepackage{amssymb, amsmath, amsthm}
 \newtheorem{thm}{Theorem}[section]
 \newtheorem{cor}{Corollary}[section]
 
 \newtheorem{ex}{Example}[section]
 \newtheorem{rem}{Remark}[section]
\title{Representation Theory of Finite Groups in Wiener--Hopf Factorization}
\author{\bf Victor Adukov }

\begin{document}
\maketitle
\begin{center}
\textit{Lenin avenue 76, Department of Mathematical and Functional
Analysis, National Research South Ural State University, Chelyabinsk
454080, Russia\\
e-mail: victor.m.adukov@gmail.com}
\end{center}

\begin{abstract}
We consider the Wiener--Hopf factorization problem for a matrix
function that is completely defined by its first column: the
succeeding columns are obtained from the first one by means of a
finite group of permutations. The symmetry of this matrix function
allows us to reduce the dimension of the problem. In particular, we
find some relations between its partial indices and can compute some
of the indices. In special cases we can explicitly obtain the
Wiener--Hopf factorization of the matrix function.
\end{abstract}

\section{Introduction}
Let $\Gamma $ be a simple smooth closed contour in the complex plane
$\mathbb C$ bounding the domain $D_{+} $. The complement of  $D_{+}
\cup \Gamma $ in $\overline {\mathbb C} ={\mathbb C}\cup \{\infty
\}$ will be denoted by $D_{-} $. We can assume that $0\in D_{+}$.
Let $A(t)$  be a continuous  and invertible  $n\times n$ matrix
function on $\Gamma $.

{\em A (right) Wiener--Hopf factorization} of $A(t)$ is  its
representation in the form
\begin{equation}\label{factWH}
A(t)=A_-(t)d(t)A_+(t), \ t\in \Gamma.
\end{equation}
Here $A_{\pm}(t)$ are continuous and invertible  matrix functions on
$\Gamma $ that admit analytic continuation into  $D_{\pm}$ and their
continuations $A_{\pm}(z)$ are invertible into these domains;
$d(t)={\rm diag}[t^{\rho_1},\ldots,t^{\rho_n}]$, where integers
$\rho_1,\ldots,\rho_n$ are called {\em the (right) partial indices}
of $A(t)$ \cite{CG,LS}.

In general, a matrix function $A(t)$ with continuous entries does
not admit the Wiener--Hopf factorization. Let $\mathfrak A$ be an
algebra of continuous functions on $\Gamma$ such that any invertible
matrix function $A(t)\in G\mathfrak A^{n\times n}$, $n\geq 1$,
admits the Wiener--Hopf factorization with the factors
$A_{\pm}(t)\in G\mathfrak A^{n\times n}$. Here $G\mathfrak
A^{n\times n}$ is the group of invertible elements of the algebra
$n\times n$ matrix functions with entries in $\mathfrak A$. Basic
examples of such algebras are the Wiener algebra $W(\mathbb T)$ (or
more generally, a decomposing $R$-algebra) and the algebra
$H_\mu(\Gamma)$ of H\"older continuous functions on $\Gamma $
(\cite{CG}, Ch.2).

For $\mathfrak A= W(\mathbb T)$ or $H_\mu(\Gamma)$ the scalar
problem can be solved explicitly. In the matrix case explicit
formulas for the factors $A_{\pm}(t)$ and the partial indices of an
arbitrary matrix function do not obtain. Therefore,  it  is
interesting to find  classes  of  matrix  functions for  which  an
explicit construction of the factorization  is possible.

In the present work we consider the factorization problem for the
algebra $\mathfrak{A}\bigotimes\mathbb{C}[G]$, where $\mathbb{C}[G]$
is the group algebra of a finite group $G=\{g_1=e, g_2, \ldots,
g_n\}$. In other word, we consider matrix functions of the form
\begin{equation}\label{matrixA}
A(t)= \begin{pmatrix}
a(g_1)&a(g_1g^{-1}_2)&\cdots&a(g_1g^{-1}_n)\\
a(g_2)&a(g_2g^{-1}_2)&\cdots&a(g_2g^{-1}_n)\\
\vdots&\vdots&\ddots&\vdots\\
a(g_n)&a(g_ng^{-1}_2)&\cdots&a(g_ng^{-1}_n)\\
\end{pmatrix}.
\end{equation}
Here $a(g)\in \mathfrak A$. This matrix  is obtained from the first
column by means of a finite group of permutations that is isomorphic
to $G$. For example, if $G$ is a cyclic group with a generator
$\zeta$, an enumeration of $G$ is $\{e, \zeta,\ldots,\zeta^{n-1}\}$,
and $a_j=a(\zeta^j)$; then
$$
A(t)= \begin{pmatrix}
a_1(t)&a_{n-1}(t)&\cdots&a_2(t)\\
a_2(t)&a_1(t)&\cdots&a_3(t)\\
\vdots&\vdots&\ddots&\vdots\\
a_n(t)&a_{n-2}(t)&\cdots&a_1(t)\\
\end{pmatrix}
$$
is a circulant matrix function.

Matrix functions of the form~(\ref{matrixA}) possesses of specific
symmetry. Hence, it is natural to expect that an application of the
representation theory for finite groups allows to reduce the
dimension $n$ of the problem. Really, it turns out that $A(t)$ can
be explicitly reduced to a block diagonal form by a constant linear
transformation.

Knowledge of the degrees $n_j$ of inequivalent irreducible unitary
representations of $G$ allows to find the dimensions and
multiplicities of these blocks. For example, let $G=A_4$. The order
$|G|$ of the group is $12$ and the number $s$ of irreducible
representations equals $4$. From the well-known relation
$n_1^2+\ldots+n_s^2=|G|$ it follows that $n_1=n_2=n_3=1, \,n_4=3$.
Thus, the $12$-dimensional factorization problem can be reduce to
three scalar problems and to a $3$-dimensional problem of
multiplicity $3$. If $G$ is abelian, then the factorization problem
is reduced to $n$ one-dimensional problems.

Also we consider the factorization problem in the algebra
$\mathfrak{A}\bigotimes Z\mathbb{C}[G]$, where $Z\mathbb{C}[G]$ is
the center of $\mathbb{C}[G]$. In this case the problem is
explicitly reduced to scalar problems.

\section{Main results}
\paragraph{1. The factorization in the algebra $\mathfrak{A}\bigotimes\mathbb{C}[G]$.}
Let $G$ be a finite group of order $|G|=n$ with identity $e$. We fix
an enumeration of $G$, $G=\{g_1=e,g_2,\ldots,g_n\}$. Let ${\mathbb
C}[G]$ be the group algebra, i.e. an inner product space of formal
linear combinations $a=\sum_{g\in G}a(g)\,g$ of $g\in G$ with
coefficients $a(g)$ in ${\mathbb C}$. The inner product is defined
by

\begin{equation}\label{innerpro}
(a,b)=\frac{1}{|G|}\sum_{g\in G}a(g)\,\overline{b(g)}.
\end{equation}

The group $G$ is embedded into ${\mathbb C}[G]$ by identifying an
element $g$ with the linear combination $1\cdot g$. Then
$g_1,\ldots,g_n$ is a basis of the linear space ${\mathbb C}[G]$.
The group operation in $G$ defines a multiplication of the elements
of the basis, and  ${\mathbb C}[G]$  is endowed  by the structure of
an algebra over the field ${\mathbb C}$. We can also consider
${\mathbb C}[G]$ as the algebra of functions $a(g)$ with the
convolution as multiplication.

Let $\mathcal{A}$ be the operator of multiplication by an element
$a=\sum_{g\in G}a(g)\,g$ in the space ${\mathbb C}[G]$. The matrix
of $\mathcal A$ with respect to the basis  $g_1,\ldots,g_n$ has the
form~(\ref{matrixA}).

We can identify the group algebra ${\mathbb C}[G]$ with the algebra
of matrices of  the form~(\ref{matrixA}).

Let $\mathfrak A$ be an algebra over $\mathbb C$  with identity $I$.
Denote by $\mathfrak{A}\bigotimes\mathbb{C}[G]$  the algebra of
matrices of  the form~(\ref{matrixA}), where $a(g_k)\in\mathfrak A$.

Let $s$ be the number of conjugacy classes of $G$,
$\{\Phi_1,\ldots,\Phi_s\}$ a set of inequivalent irreducible unitary
representations of $G$, and $n_1,\ldots,n_s$ their degrees. Let
$V_k$ be the representation space of $\Phi_k$, $k=1,\ldots,s$. We
pick some orthonormal basis of $V_k$. Let $\varphi_{ij}^k(g)$, $i,j
= 1,\ldots,n_k$, denote the matrix elements of the operator
$\Phi_k(g)$, $g\in G$,  with respect to this basis. By
$\varphi_k(g)$ we denote the matrix of $\Phi_k(g)$.

The functions
\[
\left\{ \sqrt{n_k}\,\varphi_{ij}^k(g)\ \bigl|\ 1\leq k \leq s,\ 1\leq i,j \leq n_k\right\}
\]
form the orthonormal basis of $\mathbb{C}[G]$ relative to the inner product~(\ref{innerpro})
(see \cite{S}, Proposition 4.2.11). This means that Shur's orthogonality relations
\begin{equation}\label{ShRel}
\frac{1}{|G|}\sum_{g\in G}\sqrt{n_k n_m}\,\varphi_{ij}^k(g)\overline{\varphi_{i'j'}^m(g)}=
\begin{cases}
1&\text{if $k=m$, $i=i'$, $j=j'$,}\\
0&\text{else}
\end{cases}
\end{equation}
hold.

In what follows, it is convenient to deal with a block form of matrices. Let
\[
f_k(g)=\sqrt{n_k}\,{\rm col}
\left(\varphi_{11}^k(g),\ldots,\varphi_{n_k1}^k(g);\ldots;\varphi_{1n_k}^k(g),\ldots,\varphi_{n_kn_k}^k(g)\right)
\]
denotes the column consisting of all elements of the matrix
$\sqrt{n_k}\varphi_k(g)$.

Now we form the block matrices
\[
\mathcal{F}_k=\begin{pmatrix}
f_k(g_1)\,\dots\,f_k(g_n)
\end{pmatrix}
\mbox{\quad and \quad}
\mathcal{F}=\frac{1}{\sqrt{n}}\begin{pmatrix}
\mathcal{F}_1\\
\vdots\\
\mathcal{F}_s\\
\end{pmatrix},
\]
where $n=|G|$.
The size of $\mathcal{F}_k$ is $n_k^2\times n$ and, since $n_1^2+\ldots+n_s^2=|G|$
(see, e.g., \cite{S}, Corollary~4.4.5), $\mathcal{F}$ is a $n\times n$ matrix.

The orthogonality relations~(\ref{ShRel}) mean that the matrix $\mathcal{F}$ is unitary,
i.e.
\begin{equation}\label{unitary}
\mathcal{F}\mathcal{F}^{*}=E_n,
\end{equation}
where $\mathcal{F}^{*}$ is conjugate transpose of $\mathcal{F}$ and
$E_n$ is the  identity $n\times n$ matrix. An equivalent formulation
of~(\ref{unitary}) is  the following relation
\begin{equation}\label{unitary1}
\frac{1}{n}\sum_{g\in G}f_k(g)f^*_m(g)=
\begin{cases}
E_{n_k^2},&\text{if $k=m$},\\
0,&\text{if $k\neq m$},
\end{cases}
\quad 1\leq k,m\leq s.
\end{equation}

\begin{thm}\label{Th1}
For any matrix $A\in \mathfrak{A}\bigotimes \mathbb{C}[G]$ the following factorization
\begin{equation}\label{fact1}
A=\mathcal{F}^{*}\Lambda\mathcal{F}
\end{equation}
holds.
Here $\Lambda$ is the block diagonal matrix ${\rm diag}\left[\Lambda_1,\ldots,\Lambda_s\right]$
and the $n_k^2\times n_k^2$ matrix $\Lambda_k$ is also block diagonal matrix of the form
\[
\Lambda_k={\rm
diag}\Bigl[\underbrace{\lambda_k,\ldots,\lambda_k}_{n_k}\Bigr],
\]
where $\lambda_k=\sum_{g\in G}a(g)\varphi_k(g)$, $k=1,\ldots,s$.
\end{thm}
\begin{proof}
Let us compute $\Lambda=\mathcal{F}A\mathcal{F}^{*}$. Taking into account the block structure
of $\mathcal{F}$, $\mathcal{F}^{*}$, we get the following block form of $\Lambda$:
\[
\mathcal{F}A\mathcal{F}^{*}=\frac{1}{n}\Bigl(\mathcal{F}_k\,A\,\mathcal{F}_m^*\Bigr)_{k,m=1}^s.
\]

By definition the matrices $\mathcal{F}$, $A$, $\mathcal{F}^{*}$, we have
\[
\mathcal{F}_k\,A\,\mathcal{F}_m^*=\sum_{i,j=1}^nf_k(g_i)a(g_ig_j^{-1})f_m^*(g_j).
\]
After change of the variable we obtain
\[
\mathcal{F}_k\,A\,\mathcal{F}_m^*=\sum_{g\in G}a(g)\sum_{h\in G}f_k(gh)f_m^*(h).
\]

Since $\Phi_k$ is a homomorphism, it follows that
$\varphi_k(gh)=\varphi_k(g)\varphi_k(h)$. Hence, $f_k(gh)={\rm
diag}\Bigl[\underbrace{\varphi_k(g),\ldots,\varphi_k(g)}_{n_k}\Bigr]f_k(h)$
and
\[
\mathcal{F}_k\,A\,\mathcal{F}_m^*={\rm
diag}\Bigl[\underbrace{\lambda_k,\ldots,\lambda_k}_{n_k}\Bigr]
\sum_{h\in G}f_k(h)f_m^*(h).
\]

By~(\ref{unitary1}) now we have
\[
\frac{1}{n}\mathcal{F}_k\,A\,\mathcal{F}_m^*=
\begin{cases}
\Lambda_k, &\text{if $k=m$},\\
0,& \text{if $k\neq m$}.
\end{cases}
\]

Thus, the matrix $\Lambda$ has the form as claimed.
\end{proof}
\begin{rem}
Let $M_k$ be any $n_k\times n_k$ matrix with entries $m^k_{ij}\in\mathfrak{A}$,  $k=1,\ldots,s$.
Define the functions $a(g)=\frac{1}{|G|}\sum_{k=1}^s\sum_{i,j=1}^{n_k}n_k\,m^k_{ij}\overline{\varphi_{ij}^k(g)}$.

It is not difficult to proof that
\[
\sum_{g\in G}a(g)\varphi_k(g)=M_k.
\]
Thus,
the corresponding matrix $A$ has the prescribed diagonal form:
\[
\Lambda={\rm
diag}\Bigl[\underbrace{M_1,\ldots,M_1}_{n_1};\ldots;\underbrace{M_s\ldots
M_s}_{n_s}\Bigr].
\]
Hence the algebra $\mathfrak{A}\bigotimes \mathbb{C}[G]$ is
isomorphic to the direct product $\bigl(\mathfrak{A}\bigotimes
\mathfrak{M}_{n_1}\bigr) \times\ldots\times
\bigl(\mathfrak{A}\bigotimes \mathfrak{M}_{n_s}\bigr)$, where
$\mathfrak{M}_{\ell}$ is the complete algebra of $\ell\times \ell$
matrices. For the group algebra $\mathbb{C}[G]$ this statement is
Wedderburn's theorem (see, \cite{S}, Theorem~5.5.6).

However, for our purposes it is required an explicit reduction of
the matrix~(\ref{matrixA}) to the block diagonal form. It is done in
Theorem~\ref{Th1}.
\end{rem}

Now we apply the theorem to the Wiener--Hopf factorization problem.
Let  $\mathfrak A$ be an algebra of continuous functions on the
contour $\Gamma$ such that any invertible element
$\lambda(t)\in\mathfrak A$ admits the Wiener--Hopf factorization
$$
\lambda(t)=\lambda^-(t)t^{\rho}\lambda^+(t), \ \rho={\rm
ind}_{\Gamma}\lambda(t),
$$
and any invertible matrix function with entries from $\mathfrak A$
admits the matrix Wiener--Hopf factorization~(\ref{factWH}). An
element $a(g)\in\mathfrak A$ now is a function on $\Gamma$ and we
will used the notation $a_g(t)$ for $a(g)$.

The symmetry of the matrix function $A(t)$
allows us to reduce the dimension of the Wiener--Hopf factorization problem. By Theorem~\ref{Th1} we have

\begin{cor}
 Let $A(t)$ be an invertible matrix function of  the form (\ref{matrixA}).
Then the Wiener--Hopf factorization problem for $A(t)$ is explicitly reduced to the $s$ problems
for $n_k\times n_k$ matrix functions
$$
\lambda_k(t)=\sum_{g\in G}a_g(t)\varphi_k(g),\ k=1,\ldots,s,
$$
where $\varphi_k(g)$ are the matrices of irreducible representations of $G$. Each partial
index of $\lambda_k(t)$ is the partial index of $A(t)$ of multiplicity $n_k$.

If $G^\prime$ is the commutator subgroup of $G$, then $A(t)$ has $[G:G^\prime]$ partial indices
that can be found explicitly. Here $[G:G^\prime]$ is the index of $G^\prime$.
\end{cor}
\begin{proof}
The first claim directly follows from (\ref{fact1}). Since $[G:G^\prime]$ coincides with
the number of one-dimensional representations (see, e.g., \cite{S}, Lemma~6.2.7), the second statement also holds.
\end{proof}

\paragraph{2. The factorization in the algebra $\mathfrak{A}\bigotimes Z\mathbb{C}[G]$}
Here we present the results of the work~\cite{A} (for an abelian
case see also~\cite{A1}).  We replace the group algebra
$\mathbb{C}[G]$ by its center $Z\mathbb{C}[G]$ and obtain a special
kind of matrix functions for which the Wiener--Hopf factorization
can be constructed explicitly. Actually, we are dealing with the
factorization of some special class of functionally commutative
matrix functions. It is known (see, e.g., \cite{LS}) that a
functionally commutative matrix function by a constant linear
transformation can be reduced to a triangular form and partial
indices of this matrix coincide with indices of its characteristic
functions. However, in our case the matrix can be explicitly reduced
to a diagonal form. For this purpose  it is only necessary to know
characters of irreducible representations of the group $G$.

We fix some enumeration of the conjugacy classes of $G$:
$K_1=\{e\},K_2,\allowbreak\ldots,K_s$. Let $h_j=|K_j|$ be the
conjugacy class order. The center $Z\mathbb{C}[G]$ of the group
algebra $\mathbb{C}[G]$ consists of class functions, i.e. the
functions $a(g)\in \mathbb{C}[G]$ that are constant on conjugacy
classes. In particular, the character $\chi$ of every representation of
the group $G$ is a class function. Hence $\chi(K_j)=\chi(g_j)$,
where $g_j$ is an arbitrary representative of the class $K_j$.

The indicator functions of the conjugacy classes, in other words,
$C_j=\sum_{g\in K_j}g$, $j=1,\ldots,s$, form a basis of the
commutative algebra $Z\mathbb{C}[G]$. Hence
\[
C_iC_j=\sum_{m=1}^nc_{ij}^mC_m,
\]
where $c_{ij}^m$ are structure coefficients of the algebra $Z\mathbb{C}[G]$.

In the space $Z\mathbb{C}[G]$ we consider the operator $\mathcal A$ of multiplication by
$a=\sum_{i=1}^sa_iC_i\in Z\mathbb{C}[G]$. The  matrix $A$ of the operator $\mathcal A$ with respect to the
basis $C_1,\ldots,C_s$ is defined by the formula
\begin{equation}\label{matrixAcenter}
\left(A\right)_{mj}=\sum_{i=1}^sa_ic_{ij}^m.
\end{equation}

In particular, if $G$ is an abelian group, then $s=n$ and $A$
coincides with the matrix~(\ref{matrixA}).

We  identify  $Z{\mathbb C}[G]$ with the algebra of
matrices of   the form~(\ref{matrixAcenter}). Denote by
$\mathfrak{A}\bigotimes Z\mathbb{C}[G]$  the algebra of matrices of
the form~(\ref{matrixAcenter}), where $a_i\in\mathfrak A$.

\begin{thm} \label{Th2}
Let $\chi_1,\ldots,\chi_s$ be the characters of irreducible complex
representations of the group $G$ and $n_1,\ldots,n_s$ the degrees of
these representations. Define the matrix $\mathcal F$ by the formula
\[
\mathcal{F} =\frac{1}{\sqrt{|G|}}\begin{pmatrix}
h_1\chi_1(K_1)I&h_2\chi_1(K_2)I&\ldots&h_s\chi_1(K_s)I\\
h_1\chi_2(K_1)I&h_2\chi_2(K_2)I&\ldots&h_s\chi_2(K_s)I\\
\vdots&\vdots&\ddots&\vdots\\
h_1\chi_s(K_1)I&h_2\chi_s(K_2)I&\ldots&h_s\chi_s(K_s)I\\
\end{pmatrix}.
\]

Then $\mathcal F$ is an invertible matrix, the matrix
\[
\mathcal{F}^{-1} =\frac{1}{\sqrt{|G|}}\begin{pmatrix}
\overline{\chi_1(K_1)}I&\overline{\chi_2(K_1)}I&\ldots&\overline{\chi_s(K_1)}I\\
\overline{\chi_1(K_2)}I&\overline{\chi_2(K_2)}I&\ldots&\overline{\chi_s(K_2)}I\\
\vdots&\vdots&\ddots&\vdots\\
\overline{\chi_1(K_s)}I&\overline{\chi_2(K_s)}I&\ldots&\overline{\chi_s(K_s)}I\\
\end{pmatrix}
\]
is its inverse, and for a matrix $A\in \mathfrak{A}\bigotimes Z\mathbb{C}[G]$ the following factorization
\[
A=\mathcal{F}^{-1}\Lambda\mathcal{F}
\]
holds.

Here $\Lambda={\rm diag}\left[\Lambda_1,\ldots,\Lambda_s\right]$,
$\Lambda_j=\frac{1}{n_j}\sum_{g\in G}a(g)\chi_j(g)$, $j=1,\ldots,s$.

\end{thm}
\begin{proof} Let us first find
\[
\left(\mathcal{F}\mathcal{F}^{-1}\right)_{ij}=\frac{1}{|G|}\sum_{m=1}^s h_m\chi_i(K_m)\overline{\chi_j(K_m)}I=
\frac{1}{|G|}\sum_{g\in G}\chi_i(g)\overline{\chi_j(g)}I.
\]

By the fourth character relation (see \cite{vdW}, Section 14.6), we have
\begin{equation}\label{4ChReel}
\sum_{g\in G}\chi_i(g)\overline{\chi_j(g)}=
\begin{cases}
|G|,&i=j\\
0,&i\ne j
\end{cases}.
\end{equation}
Hence $\left(\mathcal{F}\mathcal{F}^{-1}\right)_{ij}=\delta_{ij}I$ and
$\mathcal{F}^{-1}$ is the inverse of $\mathcal{F}$.

Now we compute $\mathcal{F}A\mathcal{F}^{-1}$.  From  (\ref{matrixAcenter}) and the definition
of $\mathcal{F}$, $\mathcal{F}^{-1}$ it follows  that
\[
\left(\mathcal{F}A\mathcal{F}^{-1}\right)_{ij}=
\frac{1}{|G|}\sum_{l,k,m=1}^s h_m a_k c_{kl}^m\chi_i(K_m)\overline{\chi_j(K_l)}.
\]
Since
\[
\sum_{m=1}^sh_mc_{kl}^m\chi_i(K_m)=\frac{h_k
h_l}{n_i}\chi_i({K}_k)\chi_i(K_l)
\]
(the second character relation, \cite{vdW}, Section 14.6), we obtain
\[
\left(\mathcal{F}A\mathcal{F}^{-1}\right)_{ij}=
\frac{1}{n_i}\sum_{k=1}^s  a_kh_k \chi_i({K}_k) \frac{1}{|G|}\sum_{l=1}^s h_l\chi_i(K_l)\overline{\chi_j(K_l)}.
\]
We can transform the second sum to the following form
\[
\frac{1}{|G|}\sum_{l=1}^s h_l\chi_i(K_l)\overline{\chi_j(K_l)}=
\frac{1}{|G|}\sum_{g\in G}\chi_i(g)\overline{\chi_j(g)}.
\]
Here we permit $g$ to run over all group elements since the characters are class functions.
Applying the fourth character relation to this sum we get
\[
\left(\mathcal{F}A\mathcal{F}^{-1}\right)_{ij}=
\frac{\delta_{ij}}{n_i}\sum_{k=1}^s  a_k h_k \chi_i({K}_k)=\frac{\delta_{ij}}{n_i}\sum_{g\in G}  a(g)\chi_i(g).
\]
Thus, $\mathcal{F}A\mathcal{F}^{-1}=\Lambda$.
\end{proof}

Now let $\mathfrak A$ be an algebra of continuous functions on the
contour $\Gamma$ as in the previous subsection. The matrix function
$A(t)$ is invertible if and only if the functions
$$
\Lambda_j(t)=\frac{1}{n_j}\sum_{g\in G}a_g(t)\chi_j(g), \ j=1,\ldots,s,
$$
non-vanish on $\Gamma$. Let
$\Lambda_j(t)=\Lambda_j^-(t)t^{\rho_j}\Lambda_j^+(t)$ be the
Wiener--Hopf factorization of $\Lambda_j(t)$. Here $\rho_j={\rm
ind}_{\Gamma}\Lambda_j(t)$. Then
\[
\Lambda(t)={\rm diag}[\Lambda_1^-(t),\ldots,\Lambda_s^-(t)]\cdot
{\rm diag}[t^{\rho_1},\ldots,t^{\rho_s}]\cdot
{\rm diag}[\Lambda_1^+(t),\ldots,\Lambda_s^+(t)]
\]
is the Wiener--Hopf factorization of $\Lambda(t)$.

Theorem \ref{Th2} now leads to the following result:
\begin{cor}
 Let $A(t)$ be an invertible matrix function of  the form (\ref{matrixAcenter}).
Then its Wiener--Hopf factorization
$$
A(t)=A_-(t)d(t)A_+(t)
$$
can be constructed by the formulas
\[
A_-(t)=\frac{1}{\sqrt{|G|}} \begin{pmatrix}
\Lambda_1^-(t)\overline{\chi_1(K_1)}&\Lambda_2^-(t)\overline{\chi_2(K_1)}&
\ldots&\Lambda_s^-(t)\overline{\chi_s(K_1)}\\
\Lambda_1^-(t)\overline{\chi_1(K_2)}&\Lambda_2^-(t)\overline{\chi_2(K_2)}&
\ldots&\Lambda_s^-(t)\overline{\chi_s(K_2)}\\
\vdots&\vdots&\ddots&\vdots\\
\Lambda_1^-(t)\overline{\chi_1(K_s)}&\Lambda_2^-(t)\overline{\chi_2(K_s)}&
\ldots&\Lambda_s^-(t)\overline{\chi_s(K_s)}
\end{pmatrix},
\]

\[
d(t)= {\rm diag}[t^{\rho_1},\ldots,t^{\rho_s}],
\]

\[
A_+(t)= \frac{1}{\sqrt{|G|}}\begin{pmatrix}
h_1\Lambda_1^+(t){\chi_1(K_1)}&h_2\Lambda_1^+(t){\chi_1(K_2)}&\ldots&h_s\Lambda_1^+(t){\chi_1(K_s)}\\
h_1\Lambda_2^+(t){\chi_2(K_1)}&h_2\Lambda_2^+(t){\chi_2(K_2)}&\ldots&h_s\Lambda_2^+(t){\chi_2(K_s)}\\
\vdots&\vdots&\ddots&\vdots\\
h_1\Lambda_s^+(t){\chi_s(K_1)}&h_2\Lambda_s^+(t){\chi_2(K_s)}&\ldots&h_s\Lambda_s^+(t){\chi_s(K_s)}
\end{pmatrix}.\tag*{\qed}
\]
\end{cor}

\section{Examples}

\begin{ex} Let $G=V_4$ be the Klein four-group and $A(t)$ has the form (\ref{matrixA}). $V_4$ is an abelian subgroup of
the symmetric group $S_4$:
\[
V_4=\left\{e,(12)(34),(13)(24),(14)(23)\right\},
\]
which is isomorphic to the direct product $C_2\times C_2$ of cyclic
groups of order 2. Hence, the matrix $A$ is a 2-level circulant
matrix, i.e. a $2\times 2$ block circulant matrix with  $2\times 2$
circulant blocks:

\[
A(t)=\left(\begin{array}{cc|cc}
a_1(t)&a_2(t)&a_3(t)&a_4(t)\\
a_2(t)&a_1(t)&a_4(t)&a_3(t)\\
\hline
a_3(t)&a_4(t)&a_1(t)&a_2(t)\\
a_4(t)&a_3(t)&a_2(t)&a_1(t)\\
\end{array}\right).
\]
\\

The character table of $V_4$ (see, e.g., \cite[Ch.8, S.5]{K},
\cite[Table 4.4]{S})
\begin{table}[htbp]
\begin{center}
\begin{tabular}{|c|c|c|c|c|}
\hline $V_4$& $\ \ e \ \ $ & $(12)(34)$& $(13)(24)$&$(14)(23)$ \\
\hline $\ \chi_1\ $& $1$& $1$& $1$&$1$ \\
\hline $\ \chi_2$& $1$& $\:-1\quad$& $1$&$\:-1\quad$ \\
\hline $\ \chi_3$& $1$& $1$& $\:-1\quad$&$\:-1\quad$ \\
\hline $\ \chi_4$& $1$& $\:-1\quad$& $\:-1\quad$&$1$ \\
\hline
\end{tabular}
\end{center}
\end{table}

\noindent defines the matrix
\[
\mathcal{F}=\frac{1}{2}\left(\begin{array}{rrrr}
$1$& $1$& $1$&$1$ \\
$1$& $-1$& $1$&$-1$ \\
$1$& $1$& $-1$&$-1$ \\
$1$& $-1$& $-1$&$1$ \\
\end{array}\right)
\]
that reduces $A(t)$ to the diagonal form with the elements:
\begin{align*}
\Lambda_1(t)=a_1(t)+a_2(t)+a_3(t)+a_4(t),&\quad \Lambda_2(t)=a_1(t)-a_2(t)+a_3(t)-a_4(t),\\
\Lambda_3(t)=a_1(t)+a_2(t)-a_3(t)-a_4(t),&\quad \Lambda_4(t)=a_1(t)-a_2(t)-a_3(t)+a_4(t).
\end{align*}
The indices of these functions are the partial indices of $A(t)$.
\end{ex}

\begin{ex} Let $G=S_3$ be the symmetric group of degree $3$. It is a non-abelian group of the  order $|G|=6$.
We will used the following enumeration of the group:
$G=\left\{e,(12),(13),(23),(123),(132)\right\}$.

1. {\em The factorization in the algebra $\mathfrak{A}\bigotimes \mathbb{C}[S_3]$}.
In this case, by~(\ref{matrixA}),  the matrix function $A(t)$ has the form
\[
A(t)=\begin{pmatrix}
a_1(t)&a_2(t)&a_3(t)&a_4(t)&a_6(t)&a_5(t)\\
a_2(t)&a_1(t)&a_6(t)&a_5(t)&a_3(t)&a_4(t)\\
a_3(t)&a_5(t)&a_1(t)&a_6(t)&a_4(t)&a_2(t)\\
a_4(t)&a_6(t)&a_5(t)&a_1(t)&a_2(t)&a_3(t)\\
a_5(t)&a_3(t)&a_4(t)&a_2(t)&a_1(t)&a_6(t)\\
a_6(t)&a_4(t)&a_2(t)&a_3(t)&a_5(t)&a_1(t)\\
\end{pmatrix}.
\]
A complete set of inequivalent irreducible unitary representations $\{\Phi_1,\Phi_2,\allowbreak\Phi_3\}$
is defined by the following table (see, e.g., \cite[Ch.8, S.2]{K})

\begin{table}[htbp]\caption{Irreducible representations of $S_3$}\label{tab1}
\begin{center}
\begin{tabular}{|c|c|c|c|c|c|c|}
\hline$S_3$ & $e $ & $(12)$& $(13)$& $(23)$ & $(123)$& $(132)$\\
\hline $\Phi_1$& $1$& $1$& $1$& $1$& $1$& $1$\\
\hline $\Phi_2$& $1$& $\:-1\quad$ & $\:-1\quad$ & $\:-1\quad$ & $1$& $1$\\
\hline $\Phi_3$& $\begin{pmatrix}1&0\\0&1\end{pmatrix}$& $\begin{pmatrix}0&1\\1&0\end{pmatrix}$&
$\begin{pmatrix}0&\varepsilon\\\varepsilon^{-1}&0\end{pmatrix}$&
$\begin{pmatrix}0&\varepsilon^{-1}\\\varepsilon&0\end{pmatrix}$&
$\begin{pmatrix}\varepsilon&0\\0&\varepsilon^{-1}\end{pmatrix}$&
$\begin{pmatrix}\varepsilon^{-1}&0\\0&\varepsilon\end{pmatrix}$\\
\hline
\end{tabular}
\end{center}
\end{table}

\noindent Here $\varepsilon=\frac{-1+i\sqrt{3}}{2}$, $n_1=n_2=1$, $n_3=2$.

Let us form the matrix $\mathcal{F}$:
\[
\mathcal{F}=\frac{1}{\sqrt{6}}\left(\begin{array}{rrcccc}
1&1&1&1&1&1\\
1&-1&\:-1\quad&\:-1\quad&1&1\\
\sqrt{2}&0&0&0&\sqrt{2}\,\varepsilon&\sqrt{2}\,\varepsilon^{-1}\\
0&\sqrt{2}&\sqrt{2}\,\varepsilon^{-1}&\sqrt{2}\,\varepsilon&0&0\\
0&\sqrt{2}&\sqrt{2}\,\varepsilon&\sqrt{2}\,\varepsilon^{-1}&0&0\\
\sqrt{2}&0&0&0&\sqrt{2}\,\varepsilon^{-1}&\sqrt{2}\,\varepsilon\\
\end{array}\right).
\]
This matrix reduces $A(t)$ to the block diagonal form
$$
\Lambda(t)={\rm diag}[\Lambda_1,\Lambda_2,\Lambda_3,\Lambda_3],
$$
where
\[
\Lambda_1(t)=\sum_{g\in S_3}a(g)\Phi_1(g)=a_1(t)+a_2(t)+a_3(t)+a_4(t)+a_5(t)+a_6(t),
\]
\[
\Lambda_2(t)=\sum_{g\in S_3}a(g)\Phi_2(g)= a_1(t)-a_2(t)-a_3(t)-a_4(t)+a_5(t)+a_6(t),
\]
\begin{multline*}
\Lambda_3(t)=\frac{1}{\sqrt{2}}\sum_{g\in S_3}a(g)\Phi_3(g)=\\
\begin{pmatrix}
a_1(t)+\varepsilon a_5(t)+\varepsilon^{-1}a_6(t)&a_2(t)+\varepsilon a_3(t)+\varepsilon^{-1}a_4(t)\\
a_2(t)+\varepsilon^{-1} a_3(t)+\varepsilon a_4(t)&a_1(t)+\varepsilon^{-1} a_5(t)+\varepsilon a_6(t)
\end{pmatrix}.
\end{multline*}

Thus, the problem of the  Wiener--Hopf factorization is reduced to the one-dimensional problems for
$\Lambda_1(t)$, $\Lambda_2(t)$ and the two-dimensional problem for $\Lambda_3$. In particular, for
the partial indices of $A(t)$ the following relations
\begin{gather*}
\rho_1={\rm ind}_\Gamma \Lambda_1(t),\quad \rho_2={\rm ind}_\Gamma \Lambda_2(t),\\
\rho_3+\rho_4={\rm ind}_\Gamma\,\det\Lambda_3(t), \quad \rho_3=\rho_5,\quad \rho_4=\rho_6
\end{gather*}
hold. If, for example, the condition $a_4(t)=-\varepsilon a_2(t)-\varepsilon^{-1} a_3(t)$ is fulfilled,
then the matrix $\Lambda_3(t)$ has a triangular form and the factorization $A(t)$ can be constructed explicitly.

2. {\em The factorization in the algebra $\mathfrak{A}\bigotimes
Z\mathbb{C}[S_3]$}. In $S_3$ a conjugace class consists of
permutations that have the same cycle type.  There are 3 conjugace
classes $K_1=\{e\}, K_2=\{(12)\},K_3=\{(123)\}$ and $h_1=1$,
$h_2=3$, $h_3=2$. The multiplication table for the elements $C_j$ of
the basis of $Z\mathbb{C}[S_3]$ is given below

\begin{table}[htbp]
\begin{center}
\begin{tabular}{|c|c|c|c|}
\hline  $S_3$ & $C_1$ & $C_2$& $C_3$\\
\hline $C_1$& $C_1$& $C_2$& $C_3$\\
\hline $C_2$& $C_2$& $3C_1+3C_2$& $2C_2$\\
\hline $C_3$& $C_3$& $2C_2$& $2C_1+C_3$\\
\hline
\end{tabular}
\end{center}
\end{table}

Let us form the matrix function $A(t)$ by~(\ref{matrixAcenter}):
\[
A(t)=\begin{pmatrix}
a_1(t)&3a_2(t)&2a_3(t)\\
a_2(t)&a_1(t)+2a_3(t)&2a_2(t)\\
a_3(t)&3a_2(t)&a_1(t)+a_3(t)\\
\end{pmatrix}.
\]
Using Table~\ref{tab1}, we can obtain the character table of $S_3$
\begin{table}[htbp]
\begin{center}
\begin{tabular}{|c|c|c|c|}
\hline$S_3$ & $\ \ e \ \ $ & $\{(12)\}$& $\{(123)\}$\\
\hline $\ \chi_1\ $& $1$& $1$& $1$\\
\hline $\ \chi_2$& $1$& $\:-1\quad$& $1$\\
\hline $\ \chi_3$& $2$& $0$& $\:-1\quad$\\
\hline
\end{tabular}
\end{center}
\end{table}

\noindent Then we get
\[
\mathcal{F}=\frac{1}{\sqrt{6}}\left(\begin{array}{rrr}
$1$& $3$& $2$\\
$1$& $-3$& $2$\\
$2$& $0$& $-2$\\
\end{array}\right).
\]

Now $A(t)$ is reduced to the  diagonal form
\[
\Lambda(t)={\rm diag}[\Lambda_1,\Lambda_2,\Lambda_3],
\]
where
$\Lambda_1(t)=a_1(t)+3a_2(t)+2a_3(t)$, $\Lambda_2(t)=a_1(t)-3a_2(t)+2a_3(t)$,
$\Lambda_3(t)=a_1(t)-a_3(t)$.
The partial indices of $A(t)$ coincide with the indices of these functions.
\end{ex}

\begin{ex}
 Let $G=Q_8=\{\pm 1, \pm i, \pm j, \pm k\}$ be the group of quaternions. Here $1$ is
identity of the group, $(-1)^2=1$, and  the relations $i^2 = j^2 = k^2 = ijk = -1$ are fulfilled.
$Q_8$ is a non-abelian group of the order $|G|=8$.
We will used the following enumeration of the group:
$G=\left\{1,-1,i,-i,j,-j,k,-k\right\}$.

1. {\em The factorization in the algebra $\mathfrak{A}\bigotimes \mathbb{C}[Q_8]$.}
By the formula~(\ref{matrixA}), the matrix function $A(t)$ has the following block form
with a $2\times 2$ circulant blocks:
\[
A(t)=\left(\begin{array}{cc|cc|cc|cc}
a_1(t)&a_2(t)&a_4(t)&a_3(t)&a_6(t)&a_5(t)&a_8(t)&a_7(t)\\
a_2(t)&a_1(t)&a_3(t)&a_4(t)&a_5(t)&a_6(t)&a_7(t)&a_8(t)\\
\hline
a_3(t)&a_4(t)&a_1(t)&a_2(t)&a_8(t)&a_7(t)&a_5(t)&a_6(t)\\
a_4(t)&a_3(t)&a_2(t)&a_1(t)&a_7(t)&a_8(t)&a_6(t)&a_5(t)\\
\hline
a_5(t)&a_6(t)&a_7(t)&a_8(t)&a_1(t)&a_2(t)&a_4(t)&a_3(t)\\
a_6(t)&a_5(t)&a_8(t)&a_7(t)&a_2(t)&a_1(t)&a_3(t)&a_4(t)\\
\hline
a_7(t)&a_8(t)&a_6(t)&a_5(t)&a_3(t)&a_4(t)&a_1(t)&a_2(t)\\
a_8(t)&a_7(t)&a_5(t)&a_6(t)&a_4(t)&a_3(t)&a_2(t)&a_1(t)\\
\end{array}\right).
\]

There are four one-dimensional representations and a single two-di\-men\-sional
unitary representation of $Q_8$ (see, e.g., \cite[Ch.8, Example 8.2.4]{S}).  They
are given by the following table

\begin{table}[htbp]\caption{Irreducible representations of $Q_8$}\label{tab2}
\begin{center}
\begin{tabular}{|c|c|c|c|c|}
\hline$Q_8$ & $\:\pm 1\quad $ & $\:\pm i\quad$& $\:\pm j\quad$& $\:\pm k\quad$\\
\hline $\Phi_1$& $1$& $1$& $1$& $1$\\
\hline $\Phi_2$& $1$& $1$ & $\:-1\quad$ & $\:-1\quad$\\
\hline $\Phi_3$& $1$& $\:-1\quad$ & $1$ & $\:-1\quad$\\
\hline $\Phi_4$& $1$& $\:-1\quad$ & $\:-1\quad$ & $1$\\
\hline $\Phi_5$& $\pm\begin{pmatrix}1&0\\0&1\end{pmatrix}\quad$&
$\pm\begin{pmatrix}i&0\\0&-i\end{pmatrix}\quad$&
$\pm\begin{pmatrix}0&1\\-1&0\end{pmatrix}\quad$&
$\pm\begin{pmatrix}0&i\\i&0\end{pmatrix}\quad$\\
\hline
\end{tabular}
\end{center}
\end{table}
Therefore, we have
\[
\mathcal{F}=
\frac{1}{\sqrt{8}}\left(\begin{array}{rrrrrrrr}
1&1&1&1&1&1&1&1\\
1&1&1&1&-1&-1&-1&-1\\
1&1&-1&-1&1&1&-1&-1\\
1&1&-1&-1&-1&-1&1&1\\
\sqrt{2}&-\sqrt{2}&\sqrt{2}i&-\sqrt{2}i&0&0&0&0\\
0&0&0&0&-\sqrt{2}&\sqrt{2}&i&-i\\
0&0&0&0&\sqrt{2}&-\sqrt{2}&i&-i\\
\sqrt{2}&-\sqrt{2}&-\sqrt{2}i&\sqrt{2}i&0&0&0&0\\
\end{array}\right).
\]
The matrix $\mathcal{F}$ reduces $A(t)$ to the block diagonal form
$$
\Lambda(t)={\rm diag}[\Lambda_1,\Lambda_2,\Lambda_3,\Lambda_4,\Lambda_5,\Lambda_6],
$$
where
\[
\Lambda_1(t)=a_1(t)+a_2(t)+a_3(t)+a_4(t)+a_5(t)+a_6(t)+a_7(t)+a_8(t),
\]
\[
\Lambda_2(t)=a_1(t)+a_2(t)+a_3(t)+a_4(t)-a_5(t)-a_6(t)-a_7(t)-a_8(t),
\]
\[
\Lambda_3(t)=a_1(t)+a_2(t)-a_3(t)-a_4(t)+a_5(t)+a_6(t)-a_7(t)-a_8(t),
\]
\[
\Lambda_4(t)=a_1(t)+a_2(t)-a_3(t)-a_4(t)-a_5(t)-a_6(t)-a_7(t)-a_8(t),
\]
\[
\Lambda_5=\begin{pmatrix}
a_1(t)-a_2(t)+ia_3(t)-ia_4(t)&a_5(t)-a_6(t)+ia_7(t)-ia_8(t)\\
-a_5(t)+a_6(t)+ia_7(t)-ia_8(t)&a_1(t)-a_2(t)-ia_3(t)-ia_4(t)
\end{pmatrix}.
\]

Thus, the problem of the  Wiener--Hopf factorization for $A(t)$ is reduced to the four scalar problems
and the two-dimensional problem for $\Lambda_5$. In particular, for
the partial indices of $A(t)$ the following relations
\begin{gather*}
\rho_j={\rm ind}_\Gamma \Lambda_j(t),\ j=1,\ldots,4,\\
\rho_5+\rho_6={\rm ind}_\Gamma \det\Lambda_5(t),\quad   \rho_5=\rho_7,\ \rho_6=\rho_8
\end{gather*}
are fulfilled.

2. {\em The factorization in the  algebra $\mathfrak{A}\bigotimes Z\mathbb{C}[Q_8]$.}
There are 5
conjugace classes $K_1=\{e\}, K_2=\{ -1\}, K_3=\{\pm i\}$, $K_4=\{\pm j\}$; $K_5=\{\pm k\}$,
$h_1=h_2=1$, $h_3=h_4=h_5=2$.

The multiplication table for the elements $C_j$ of the basis of $Z\mathbb{C}[Q_8]$ has the following form

\begin{table}[htbp]
\begin{center}
\begin{tabular}{|c|c|c|c|c|c|}
\hline $Q_8$& $\ C_1\ $ & $\ C_2\ $& $C_3$&$C_4$&$C_5$\\
\hline$\ C_1\ $& $C_1$& $C_2$& $C_3$&$C_4$&$C_5$\\
\hline $C_2$& $C_2$& $C_1$& $C_3$&$C_4$&$C_5$\\
\hline $C_3$& $C_3$& $C_3$& $2C_1+2C_2$&$2C_5$&$2C_4$\\
\hline $C_4$& $C_4$& $C_4$&$2C_5$& $2C_1+2C_2$&$2C_3$\\
\hline $C_5$& $C_5$& $C_5$& $2C_4$&2$C_3$&$2C_1+2C_2$\\
\hline
\end{tabular}
\end{center}
\end{table}

Hence, by the formula~(\ref{matrixAcenter}), we obtain
\[
A(t)=\left(\begin{array}{ccccc}
a_1(t)&a_2(t)&2a_3(t)&2a_4(t)&2a_5(t)\\
a_2(t)&a_1(t)&2a_3(t)&2a_4(t)&2a_5(t)\\
a_3(t)&a_3(t)&a_1(t)+a_2(t)&a_5(t)&2a_4(t)\\
a_4(t)&a_4(t)&2a_5(t)&a_1(t)+a_2(t)&2a_3(t)\\
a_5(t)&a_5(t)&2a_4(t)&2a_3(t)&a_1(t)+a_2(t)\\
\end{array}\right).
\]

The group $Q_8$ has the following character table (see Table~\ref{tab2}):

\begin{table}[htbp]
\begin{center}
\begin{tabular}{|c|c|c|c|c|c|}
\hline $Q_8$& $\ \{1\}\ $ & $\{-1\}\ \ $& $\{\pm i\}\ \ $&$\{\pm j\}\ \ $&$\{\pm k\}\ \ $\\
\hline $\ \chi_1\ $& $1$& $1$& $1$&$1$&$1$\\
\hline $\ \chi_2$& $1$& $1$& $1$& $\:-1\quad$& $\:-1\quad$\\
\hline $\ \chi_3$& $1$& $1$& $\:-1\quad$&$1$&$\:-1\quad$\\
\hline $\ \chi_4$& $1$& $1$& $\:-1\quad$&$\:-1\quad$&$1$\\
\hline $\ \chi_5$& $2$& $\:-2\quad$& $0$&$0$&$0$\\
\hline
\end{tabular}
\end{center}
\end{table}

By Theorem~\ref{Th2}, the matrix $\mathcal{F}$
\[
\mathcal{F}=\frac{1}{\sqrt{8}}\left(\begin{array}{rrrrr}
1& 1& {2}&{2}&{2}\\
1& 1& {2}& -{2}&-{2}\\
1& 1& -{2}&{2}&-{2}\\
1& 1& -{2}&-{2}&{2}\\
2& -2& 0&0&0
\end{array}\right)
\]
reduces $A(t)$ to the diagonal form with the following diagonal elements:
\[
\begin{split}
\Lambda_1(t)&=a_1(t)+a_2(t)+2a_3(t)+2a_4(t)+2a_5(t),\\
\Lambda_2(t)&=a_1(t)+a_2(t)+2a_3(t)-2a_4(t)-2a_5(t),\\
\Lambda_3(t)&=a_1(t)+a_2(t)-2a_3(t)+2a_4(t)-2a_5(t),\\
\Lambda_4(t)&=a_1(t)+a_2(t)-2a_3(t)-2a_4(t)+2a_5(t),\\
\Lambda_5(t)&=a_1(t)-a_2(t).
\end{split}
\]
The indices of these functions are the partial indices of $A(t)$.
\end{ex}

\end{document}